\documentclass[12pt,psfig]{article}
\usepackage{amssymb}
\usepackage{amsmath,amsthm}
\usepackage{amsfonts}
\usepackage{graphicx}
\usepackage{graphics}
\usepackage[ruled,vlined]{algorithm2e}
\newtheorem{Theorem}{Theorem}[section]

\newtheorem{Lemma}[Theorem]{Lemma}

\newtheorem{Remark}[Theorem]{Remark}

\def\bfm#1{\protect{\makebox{\boldmath $#1$}}}

\def\u {\bfm{u}}

\def\c {\bfm{c}}
\def\pphi {\bfm{\phi}}
\def\p {\bfm{p}}
\newcommand{\R}{\mathbb R}
\newcommand{\N}{\mathbb N}
\newcommand{\Z}{\mathbb Z}
\newcommand{\Om}{\Omega}
\newcommand{\wh}{\widehat}
\newcommand{\wt}{\widetilde}
\newcommand{\q}{\quad}
\begin{document}
\date{\small\textsl{\today}}
\title{A meshless based method for solution of integral equations: Improving the error analysis}
\author{
\large D. Mirzaei \\
\small{Department of Mathematics, University of Isfahan, 81745-163 Isfahan, Iran.}
} 
\date{\small{March 10, 2012, Revised in August 25, 2015}}
\maketitle
\begin{abstract}
This draft concerns the error analysis of a collocation method based on the moving least squares (MLS) approximation
for integral equations, which improves the results of \cite{mirzaei-dehghan:2010-1} in the analysis part.
This is mainly a translation from Persian of some parts of Chapter 2 of the author's PhD thesis in 2011.
\\
\end{abstract}
\section{Introduction}
In \cite{mirzaei-dehghan:2010-1} a meshless method based on the {\em moving least squares} (MLS) was applied for integral equations of the second kind, and an error analysis was presented for
Fredholm integral equations. Here a more interesting presentation of the MLS approximation and its error estimation are reported, and  the analysis of the MLS collocation method for Fredholm integral equations of the second kind is revised. The analysis is mainly based on the excellent book \cite{atkinson:1997-1}.
\section{MLS approximation}
Let $\Omega\subset \R^d$, for positive integer $d$, be a nonempty and bounded set. Assume,
$$
X = \{x_1, x_2,\ldots ,x_N\}\subset \Omega,
$$
is a set containing $N$ scattered points.
The {\em fill distance} of $X$ is defined to be
\begin{equation*}
h_{X,\Omega}=\sup_{x\in\Omega}\min_{1\leqslant j\leqslant N}\|x-x_j\|_2,
\end{equation*}
and the {\em separation distance} is defined by
\begin{equation*}
q_{X}=\frac{1}{2}\min_{i\neq j}\|x_i-x_j\|_2.
\end{equation*}
A set $X$ of data sites is said to be
{\em quasi-uniform} with respect to a constant $c_{\mathrm{qu}}>0$ if
\begin{equation}\label{quasi-uniform}
q_X\leqslant h_{X,\Omega}\leqslant c_{\mathrm{qu}} q_X.
\end{equation}
Henceforth, we use $\mathbb P_m^d$, for $m\in \N_{0}=\{n\in\Z, n\geqslant 0\}$,
as the space of $d$-variable polynomials of degree at most $m$ of dimension
$Q:={m+d\choose d}$. A basis for this space is denoted by $\{p_1,\ldots,p_Q\}$ or $\{p_\alpha\}_{0\leqslant|\alpha|\leqslant m}$.

A set $X=\{x_1,\ldots,x_N\}\subset \R^d$ with $N\geqslant Q$ is called $\mathbb P_m^d$-unisolvent if the zero polynomial is the only polynomial
from $\mathbb P_m^d$ that vanishes on $X$.

The MLS provides an approximation $s_{u,X}$ of $u$ in terms of values $u(x_j)$ at
centers $x_j$ by
\begin{equation}\label{mlsapp}
u(x)\approx s_{u,X}(x) = \sum_{j=1}^N \phi_j(x) u(x_j), \quad x\in \Omega,
\end{equation}
where $\phi_j$ are \emph{MLS shape functions} given by
\begin{equation*}
\phi_j(x)=w(x,x_j)\sum_{k=1}^Q\lambda_k(x) p_k(x_j),
\end{equation*}
where the influence of the centers is governed by a weight function $w_j(x)=w(x,x_j)$, which vanishes for arguments $x,x_j\in\Omega$ with $\|x-x_j\|_2$ greater than
a certain threshold, say $\delta$. Thus we can define $w_j(x)=K((x-x_j)/\delta)$ where $K:\R^d\to\R$ is a nonnegative function with support in the unit ball $B(0,1)$. Coefficients $\lambda_k(x)$ are the unique solution of
\begin{equation*}
\sum_{k=1}^Q\lambda_k(x)\sum_{j\in J(x)}w_j(x)p_k(x_j)p_\ell(x_j)=p_\ell(x),\quad 0\leqslant \ell\leqslant Q,
\end{equation*}
where $J(x)=\{j: \|x-x_j\|_2\leqslant \delta\}$ is the family of indices of points in the support of the weight function.
In vector form
\begin{equation*}
\pphi(x) = W(x)P^T(PW(x)P^T)^{-1}\p(x),
\end{equation*}
where $W(x)$ is the diagonal matrix carrying the weights $w_j(x)$ on its diagonal, $P$ is a $Q\times \#J(x)$ matrix of values $p_k(x_j)$, $j\in J(x)$, $1\leqslant k\leqslant Q$, and $\p=(p_1,\ldots ,p_Q)^T$.
In the MLS, one finds the best approximation to $u$ at point $x$, out of $\mathbb P_m^d$ with respect to a discrete $\ell^2$ norm induced by
a \emph{moving} inner product, where the corresponding weight function depends not only on points $x_j$ but also on the evaluation point $x$ in question.
Note that, if for every $x\in\Omega$ the set $\{x_j: j\in J(x)\}$ is $\mathbb P_m^d$-unisolvent then $A(x)=PW(x)P^T$ is a symmetric positive definite matrix.
More details can be found in Chapter 4 of \cite{wendland:2005-1}.
In what follows we will assume that $K$ is nonnegative and continuous on $\R^d$ and positive on the ball $B(0,1/2)$. In many applications we can assume that
$$
K(x)=\varphi(\|x\|_2), \quad x\in \R^d,
$$
meaning that $K$ is a radial function. Here $\varphi:[0,\infty)\to \R$ is positive on $[0,1/2]$, supported in $[0,1]$ and its even extension is nonnegative and continuous on $\R$. If we assume that $K\in C^k(\R^d)$ then $\phi_j\in C^n(\Omega)$ where $n = \min\{k,m\}$.
This implies that $s_{u,X}\in C^n(\Omega)$.

It is well-known that \cite{wendland:2005-1} if $X=\{x_1,\ldots,x_N\}$ is a quasi-uniform set in $\Omega\subset\R^d$, where $\Omega$ is
a compact set and satisfies an interior cone condition,
then the MLS shape functions $\{\phi_j\}$
provide a {\em stable local polynomial reproduction} of degree $m\in \N_0$ on $\Omega$, i.e.
there exist constants
$h_0,C_{1},C_{2}>0$ independent of $X$ such that for every $x\in\Omega$
\begin{enumerate}
\item $ \sum_{j=1}^N\phi_{j}(x) p(x_j)=p(x),
\,\forall p\in\mathbb P_m^d,$
\item $ \sum_{j=1}^N|\phi_{j}(x)|\leqslant C_{1}$,
\item $\phi_{j}(x)=0\,\, \mbox{if}\,\, \|x-x_j\|_2>\delta=2C_{2}h_{X,\Omega}$,
\end{enumerate}
for all
$X$ with $h_{X,\Omega}\leqslant h_0$.

Note that, a set $\Omega\subset \mathbb R^d$ is said to satisfy an
interior cone condition if there exist an angle
$\theta\in(0,\pi/2)$ and a radius $r>0$ such that for every
$x\in\Omega$ a unit vector $\xi(x)$ exists such that the cone
$$
C(x,\xi,\theta,r):=\big\{ x+ty: y\in\R^d, \|y\|_2=1, y^T\xi\geqslant \cos\theta, t\in[0,r]\big\}
$$
is contained in $\Omega$.

The following theorem shows that the MLS approximation converges uniformly for continuous functions on compact domain $\Omega$. 
\begin{Theorem}\label{thm-convergence}
Suppose that $\Omega\subset\R^d$ is compact and satisfies an interior cone condition.
The MLS approximation $s_{u,X}$ converges uniformly for all continuous function $u$, as $h_{X,\Omega}$ goes to zero for quasi-uniform sets $X$.
\end{Theorem}
\begin{proof}
For a fixed $x\in\Omega$, suppose that $p_0$ is the constant polynomial with $p_0(x)=u(x)$.
The conditions of Theorem ensure that the MLS shape functions provide a stable local polynomial reproduction. Thus we can write
\begin{align*}
|u(x)-s_{u,X}(x)| & = \Big| p_0(x)-\sum_{j=1}^N \phi_j(x)u(x_j)\Big|\\
& =\Big| \sum_{j=1}^N \phi_j(x)\big(p_0(x_j)-u(x_j)\big)\Big|\\
&\leq \sum_{j=1}^N |\phi_j(x)|\big| p_0(x_j)-u(x_j)\big|\\
&\leq C_1 \|u-p_0\|_{\infty,B(x,\delta)\cap\Omega} \\
&=C_1 \max_{y\in B(x,\delta)\cap\Omega}|u(y)-u(x)|\\
&\leq C_1 \omega(u,\delta),
\end{align*}
where $\omega(u,\delta)$ is the modulus of continuity of $u$. The compactness of $\Omega$ and $\delta = ch_{X,\Omega}$ give the
uniform convergence.
\end{proof}
Finally, the following error estimation can be proved for smoother functions. Note that a domain with a Lipschitz boundary, automatically satisfies an interior cone condition. The reader is referred to
\cite{mirzaei:2015-1} for proof.
\begin{Theorem}\label{thm-errorLq}
Suppose that $\Omega\subset\R^d$ is a compact set with a Lipschitz boundary.
Let $m$ be a positive integer. If $u\in W^{m+1}_\infty(\Omega)$,
there exist constants $C > 0$ and $h_0>0$ such that
for all $X=\{x_1,\ldots ,x_N\}\subset \Omega$ with $h_{X,\Omega}\leqslant h_0$ which are quasi-uniform with the same $c_{\mathrm{qu}}$ in (\ref{quasi-uniform}), the estimate
\begin{equation}\label{LPRerrorLp}
\|u-s_{u,X} \|_{L_\infty(\Omega)} \leqslant Ch_{X,\Omega}^{m+1}\|u\|_{W^{m+1}_\infty(\Omega)}
\end{equation}
holds.
\end{Theorem}
In numerical implementation,
for computing the MLS approximation at a sample point $\wh x\in \Omega$, the {\em shifted} and {\em scaled} polynomial basis functions
$$
\left\{ \frac{(x-\wh x)^\alpha}{h_{X,\Omega}^{|\alpha|}} \right\}_{0\leq\alpha\leq m}
$$
are used as a basis for $\mathbb P_m^d$. In fact we change the basis functions as the evaluation point is changed. This leads to a more stable algorithm.
For example, the use of shifted and scaled basis functions overcomes the instability of the reported results in Tables 1, 4 and 6 of \cite[Section 6]{mirzaei-dehghan:2010-1}
for quadratic basis functions.
\section{The MLS collocation method}
A Fredholm integral equation of the second kind can be written as
\begin{equation}\label{5inteq}
\lambda u(x)+\int_\Om \kappa(x,s) u(s) ds =f(x),\q x\in\Om\subset\R^d,
\end{equation}
where $u$ is an unknown function, $\lambda$ is a real parameter, $\Omega$ is a compact
domain in $\R^d$, $f$ is a given continuous right-hand side function, and $\kappa$ is a given continuous kernel in $\Omega\times \Omega$.
The above integral equation can be written in the following abstract form,
\begin{equation*}
(\lambda-\mathcal F)u=f,
\end{equation*}
where
$$
\mathcal Fu=\int_\Omega\kappa(x,s)u(s)ds.
$$
We consider a set of {\em trial points}
$ X=\{x_1,x_2,\ldots,x_N\}\subset \Om$
with fill distance
$h_{X,\Om}$.
Regarding the previous section, we assume that $X$ is a {quasi-uniform} set and admits a well-defined MLS approximation. Suppose that
$\phi_1,\ldots,\phi_N$
are the MLS shape functions constructed by polynomial space $\mathbb P_m^d$ and weight function $K\in C^k(\R^d)$, $k\in \N_0$. Define
\begin{equation*}
V_N:=\mbox{span} \{\phi_1,\ldots,\phi_N\},
\end{equation*}
as a finite dimensional subspace of $C(\Omega)$.
According to \eqref{mlsapp}, the MLS approximation
$\wh u:=s_{u,X}$
of
$u$
is
\begin{equation*}
u\approx \wh u=\sum_{j=1}^N \phi_j u(x_j) \in V_N.
\end{equation*}
Moreover, we define a projection operator
$\mathcal P_N: C(\Omega)\mapsto V_N$	
which interpolates any continuous function into $V_N$
on {\em test points}
$Y=\{y_1,\ldots,y_M\}\subset \Om$ .
More precisely, for all
$u\in C(\Omega)$
we define
\begin{equation*}
\mathcal P_N u:=\sum_{j=1}^N \phi_j c_j,\; \mbox {with}\; \mathcal P_Nu(y_i)=u(y_i),\, 1\leq i\leq M.
\end{equation*}
In what follows, we let $M=N$ and we assume that, the distribution of both sets of test and trial points are well enough to ensure the non-singularity of $\Phi_N =\big(\phi_j(y_k)\big)_{i,j=1}^N$. If it happens then
$\mathcal P_N$ is well-defined.
Since $\wh u\in V_N$, we simply have
$ \mathcal P_N\wh u=\wh u$.
Replacing $u$ by $\wh u$ in
(\ref{5inteq})
we get
\begin{equation*}
\sum_{j=1}^N\left[\lambda \phi_j(x)+\int_\Om \kappa(x,s)
\phi_j(s) ds\right] u(x_j) = f(x)+ r(x),
\end{equation*}
where
$r(x)$
is the reminder. In the {\em collocation} method we assume that the reminder is vanished at test points $Y$, i.e.
$$
\mathcal P_N r = 0,
$$
which leads to
\begin{equation*}
\sum_{j=1}^N\left[\lambda \phi_j(y_i)+\int_\Om \kappa(y_i,s)
\phi_j(s) ds\right] u(x_j) = f(y_i),\q 1\leq i\leq N,
\end{equation*}
or in an abstract form
\begin{equation*}
\mathcal P_N(\lambda-\mathcal F)\wh u = \mathcal P_Nf.
\end{equation*}
According to the property
$ \mathcal P_N\wh u=\wh u$, we have
\begin{equation*}
(\lambda-\mathcal P_N\mathcal F)\wh u = \mathcal P_Nf.
\end{equation*}
The involved integral can be treated by a numerical quadrature of the form
\begin{equation}\label{5quadrature}
\int_\Om g(s)\,ds\approx \sum_{k=1}^{Q_N}g(\tau_k)\omega_k,\q g\in C(\Om),
\end{equation}
where
$\{\tau_k\}$
and
$\{\omega_k\}$, for $1\leq k\leq Q_N$
are integration points and weights, respectively.
We assume that for all
$g\in C(\Om)$
the quadrature converges to the exact value of integral as $Q_N$ increases.
Now we define
\begin{equation}\label{5integraloper}
\mathcal F_Nu(x):=\sum_{k=1}^{Q_N}\kappa(x,\tau_k)u(\tau_k)\omega_k,\q x\in\Om,\, u\in C(\Om).
\end{equation}
If we replace $\mathcal F \wh u$ by $\mathcal F_N \wh u$,
we will get
\begin{equation}\label{5eqapprox3}
\sum_{j=1}^N\left[\lambda \phi_j(y_i)+\sum_{k=1}^{Q_N} \kappa(y_i,\tau_k)
\phi_j(\tau_k) \omega_k\right] \wt u_j = f(y_i),\q 1\leq i\leq N,
\end{equation}
where $\wt u_j$ are the approximation values of $u(x_j)$. Solving the linear system of equations \eqref{5eqapprox3}
gives the values $\wt u_j,\, j=1,\ldots,N$, and
finally one can approximate
\begin{equation*}
u(x)\approx u_N(x)=\sum_{j=1}^N \phi_j(x)\wt u_j,
\end{equation*}
for any $x\in\Om$.
The abstract form of equation
(\ref{5eqapprox3})
is
\begin{equation*}
\mathcal P_N(\lambda-\mathcal F_N) u_N = \mathcal P_N f.
\end{equation*}
Since
$u_N\in V_N$,
we have
$\mathcal P_N u_N=u_N$ and
the above equation can be rewritten as
\begin{equation}\label{5eqapprox3comp}
(\lambda-\mathcal P_N\mathcal F_N) u_N = \mathcal P_N f
\end{equation}
which shows that the scheme is a {\em discrete collocation method} \cite{atkinson:1997-1}.
Consequently an {\em iterated discrete collocation} solution can be obtained. For this purpose  we set
\begin{equation}\label{5iterated}
v_N(x)=\frac{1}{\lambda}[f(x)+\mathcal F_Nu_N(x)], \q \forall x\in\Om,
\end{equation}
and by applying the operator
$\mathcal P_N$
on both sides of
(\ref{5iterated}),
and using the relation
(\ref{5eqapprox3comp})
we simply have
$$
\mathcal P_N v_N=u_N.
$$
Thus we conclude
\begin{equation}\label{5iteratedcomp}
(\lambda-\mathcal F_N\mathcal P_N)v_N=f.
\end{equation}
Equations
(\ref{5eqapprox3comp})
and
(\ref{5iteratedcomp})
will be referred in the next section when we will try to give the error bounds for $u-u_N$ and $u-v_N$.

Usually and in this paper the case
$M=N$ is assumed which leads to a square final linear system. In addition we can assume that
$X=Y$.
The case
$M>N$
is called {\em oversampling} which may help if there is a problem with solvability.

\section{Error Analysis}
As we discussed in the previous section, the method is a discrete collocation, and the solvability of the integral equation
(\ref{5inteq})
and some insights on integration operators
$\mathcal F_N$
and projections
$\mathcal P_N$
are required to obtain the final error bound. Moreover, an error bound for the MLS approximation should be invoked.

According to
(\ref{5integraloper})
we define
$$
\|\mathcal F_N\|:=\max_{x\in\Om}\sum_{k=1}^{K_N}|\omega_k \kappa(x,\tau_k)|.
$$
A direction which makes the analysis possible is to seek for characteristic properties of operators $\mathcal F_N$ which imply
\begin{equation}\label{5kkn}
\|(\mathcal F-\mathcal F_N)\mathcal F\|\rightarrow 0,\q \|(\mathcal F-\mathcal F_N)\mathcal F_N\|\rightarrow 0,
\; \mbox{ as }\; N\rightarrow\infty.
\end{equation}
For this, we
assume that $\{\mathcal F_N, \, N\geq1\}$
possesses the following properties:
\begin{enumerate}
\item
$\mathcal X$ is a Banach space, and
$ \mathcal F$ and
$ \mathcal F_N$,
for
$N\geq1$,
are linear operator on
$\mathcal X$
into
$\mathcal X$.
\item
$\mathcal F_Nu\rightarrow \mathcal Fu$ as $N\rightarrow \infty$,
for all
$u\in\mathcal X$.
\item
The set $\{\mathcal F_N,\, N\geq1\}$ is collectively compact which means that
$ \{\mathcal F_Nu,\, N\geq1,\; \|u\|\leq1\}$
has a compact closure in
$\mathcal X$.
\end{enumerate}
Then $\{\mathcal F_N\}$ is said to be a {\em collectively compact family of pointwise
convergent operators}. According to
\cite[Lemma 4.1.2]{atkinson:1997-1},
if
$\{\mathcal F_N, \, N\geq1\}$
is a
 collectively compact family of pointwise convergent operators,
then
(\ref{5kkn})
is satisfied.
Finally,
\cite[Theorem 4.1.1]{atkinson:1997-1}
paves the way for finding the final error bound.
\begin{Theorem}\label{5thmst}
Let
$\mathcal{X}$
be a Banach space, let
$\mathcal{S}$
and
$\mathcal{T}$
be bounded operators on
$\mathcal{X}$
to
$\mathcal{X}$ and let
$\mathcal{S}$
be compact. For given
$\lambda\neq0$,
assume
$\lambda-\mathcal T:~\mathcal{X}\longrightarrow\hspace{-.8cm}^{1-1}_{onto}\hspace{.1cm}\mathcal{X}$,
which implies
$(\lambda-\mathcal T)^{-1}$
exists as a bounded operator on
$\mathcal{X}$
to
$\mathcal{X}$. Finally assume
\begin{equation}\label{5thmsteq1}
\|(\mathcal T-\mathcal S)\mathcal S
\|<\frac{|\lambda|}{\|(\lambda-\mathcal T)^{-1}\|},
\end{equation}
then
$(\lambda-\mathcal S)^{-1}$
exists and it is a bounded operator from
$\mathcal{X}$
to
$\mathcal{X}$. In fact, we have
\begin{equation}\label{5thmsteq2}
\|(\lambda-\mathcal S)^{-1}\|\leq\frac{1+\|(\lambda-\mathcal
T)^{-1}\|\|\mathcal S\|}{|\lambda|-\|(\lambda-\mathcal
T)^{-1}\|\|(\mathcal T-\mathcal S)\mathcal S \|}.
\end{equation}
If
$(\lambda-\mathcal T)w=f$
and
$(\lambda-\mathcal S)z=f$,
then
\begin{equation}\label{5thmsteq3}
\|w-z\|\leq\|(\lambda-\mathcal S)^{-1}\|\|\mathcal Tw-\mathcal Sw\|.
\end{equation}
\end{Theorem}
Now we go back to equations \eqref{5eqapprox3comp} and
(\ref{5iteratedcomp}).
In section 3.4 of
\cite{atkinson:1997-1} it is proved that the existence of the inverse operators
$ (\lambda-\mathcal F_N\mathcal P_N)^{-1}$
and
$ (\lambda-\mathcal P_N\mathcal F_N)^{-1}$
are related to each other. If
$ (\lambda-\mathcal P_N\mathcal F_N)^{-1}$ exists, then so does
$ (\lambda-\mathcal F_N\mathcal P_N)^{-1}$ and
$$
(\lambda-\mathcal F_N\mathcal P_N)^{-1}=\frac{1}{\lambda} [I+\mathcal F_N(\lambda-\mathcal P_N\mathcal F_N)^{-1}
\mathcal P_N].
$$
Conversely, if
$ (\lambda-\mathcal F_N\mathcal P_N)^{-1}$
exists, then so does
$ (\lambda-\mathcal P_N\mathcal F_N)^{-1}$
and
$$
(\lambda-\mathcal P_N\mathcal F_N)^{-1}=\frac{1}{\lambda} [I+\mathcal P_N(\lambda-\mathcal F_N\mathcal P_N)^{-1}
\mathcal F_N].
$$
By combining these, we also have
$$
(\lambda-\mathcal P_N\mathcal F_N)^{-1}\mathcal P_N =\mathcal P_N (\lambda-\mathcal F_N\mathcal P_N)^{-1}.
$$
We can choose to show the existence of either $ (\lambda-\mathcal F_N\mathcal P_N)^{-1}$ or $ (\lambda-\mathcal P_N\mathcal F_N)^{-1}$
whichever is the more convenient, and the existence of the other inverse will
follow immediately.

To use the results of Theorem
\ref{5thmst} for schemes \eqref{5eqapprox3comp} and
(\ref{5iteratedcomp}), we should first prove that
$\{\mathcal F_N\mathcal P_N,\; N\geq1\}$
is a ``collectively compact family of pointwise convergent operators".
To this aim, we need a uniform bound for
 $\|\mathcal P_N\|$.

\begin{Lemma}\label{5lemmapn}
Assume that
$\phi_1,\ldots,\phi_N$
are the MLS shape functions on the quasi uniform set
$X=\{x_1,\ldots,x_N\}$ with fill distance
$h_{X,\Omega}$
on a compact domain
$\Omega$ which satisfies an interior cone condition.
If $\|\Phi_N^{-1}\|_\infty=\mathcal O(1)$ independent of $N$ (or $h_{X,\Om}$), then there exists a constant $c_P$ independent of $N$ such that
$ \|\mathcal P_N\|\leq c_P$, and
$\mathcal P_Nu\rightarrow u$ uniformly for all
$u\in C(\Omega)$.
\end{Lemma}
\begin{proof}
First
\begin{equation*}
\mathcal P_N u(x)=\sum_{j=1}^N \phi_j(x) c_j,\; \mbox { and }\; \mathcal P_Nu(y_i)=u(y_i),\, 1\leq i\leq N,
\end{equation*}
give
$\c=\Phi^{-1}_N\u$.
On the other hand we have
$$
\|\mathcal P_Nu\|_\infty\leq \|\c\|_\infty \max_{x\in\Omega}\sum_{j=1}^N|\phi_j(x)|\leq C_1\|\Phi^{-1}_N\|_\infty\|u\|_\infty.
$$
The last inequality is satisfied because of the $L_1$ stability of the MLS shape functions (the second property of a stable local polynomial reproduction).
Thus we can write
$$
\|\mathcal P_N\|=\sup_{u\in C(\Om)} \frac{\|\mathcal P_Nu\|_\infty}{\|u\|_\infty}\leq C_1\|\Phi^{-1}_N\|_\infty,
$$
which leads to
\begin{equation}\label{5cpbound}
c_P:=\sup_N \|\mathcal P_N\| <\infty.
\end{equation}
Finally, if $\wh u$ is the MLS approximation of $u$ on $X$ then
\begin{equation*}
\begin{split}
\|\mathcal P_Nu-u\|_\infty&\leq \|\mathcal P_Nu-\mathcal P_N\wh u\|_\infty+\|u-\wh u\|_\infty\\
&\leq (1+c_P)\|u-\wh u\|_\infty \\
&\leq C(1+c_P) \omega(u,h_{X,\Omega})
\end{split}
\end{equation*}
In the first inequality we have used
$\mathcal P_N\wh u =\wh u$, and in the last one we have applied
Theorem \ref{thm-convergence}.
Since the points are quasi uniform, $N\to \infty$ implies $h_{X,\Om}\to 0$ and
$\mathcal P_Nu\rightarrow u$, uniformly.
\end{proof}
\begin{Remark}
Experiments show that $\|\Phi^{\dag}_N\|_\infty$ is of order $1$ independent of the fill distance $h_{X,\Omega}$ even if $M=N$. But it remains to prove this assertion, theoretically.
\end{Remark}
In the following lemma we prove that under some conditions
$\{\mathcal F_N\mathcal P_N,\; N\geq1\}$ is a collectively compact family of pointwise convergent operators.
\begin{Lemma}\label{5lemmaknpn}
Assume that
$\{\mathcal F_N,\; N\geq1\}$
is a collectively compact family of pointwise convergent operators on
$\mathcal X=C(\Omega)$. Then
$\{\mathcal F_N\mathcal P_N,\; N\geq1\}$
is a collectively compact family of pointwise convergent operators on $\mathcal X$.
\end{Lemma}
\begin{proof}
From
(\ref{5cpbound}) we have
$c_P\equiv\sup\|\mathcal P_N\|<\infty$. The pointwise convergence of
$\{\mathcal F_N\}$
implies that
$c_F\equiv\sup \|\mathcal F_N\|<\infty$. Together, these imply the uniform boundedness of
$\{\mathcal F_N\mathcal P_N\}$
with a bound of
$c_Pc_F$.
For the pointwise convergence on $C (\Omega)$ we have
\begin{align*}
\|\mathcal Fu-\mathcal F_N\mathcal P_Nu \|_\infty&\leq
\|\mathcal Fu-\mathcal F_Nu \|_\infty+\|\mathcal F_N(u-\mathcal P_Nu) \|_\infty\\
&\leq \|\mathcal Fu-\mathcal F_Nu \|_\infty+c_F\|u-\mathcal P_Nu \|_\infty,
\end{align*}
and the convergence now follows from that of
$\{\mathcal F_Nu\}$
and
$ \{\mathcal P_Nu\}$.
To prove the collective compactness of
$\{\mathcal F_N\mathcal P_N\}$
we must show that
$$
\mathcal K=\{\mathcal F_N\mathcal P_Nu\,:\, N\geq1,\; \|u\|_\infty\leq1\}
$$
has a compact closure in
$C(\Omega)$.
From
(\ref{5cpbound})
we have
$$
\mathcal K\subset\{\mathcal F_Nu\,:\, N\geq1,\; \|u\|_\infty\leq c_P\}
$$
which proves the assertion because
$\{\mathcal F_N\}$ is collectively compact.
\end{proof}
Now
(\ref{5kkn}) is satisfied by replacing $\mathcal F_N$ by
$\mathcal F_N\mathcal P_N$, and we can apply Theorem
\ref{5thmst}.
\begin{Theorem}\label{5thmfinal}
Let
$\Omega\subset\R^d$
be a compact set with a Lipschitz boundary, and  the quasi uniform set $X=\{x_1,\ldots,x_N\}\subset \Omega$ be a set of trial points with fill distance $h_{X,\Omega}$.
Let
$\{\mathcal P_N\}$ be a family of interpolant operators from
$C(\Omega)$ to $V_N=\mathrm{span}\{\phi_1,\ldots,\phi_N\}$ on test points $Y=\{y_1,\ldots,y_N\}\subset\Omega$,
where $\phi_j$ are the MLS shape functions based on $X$ and polynomial space $\mathbb P_m^d$.
Assume that the distribution of points is well enough to ensure the non-singularity of $\Phi_N$, and $\|\Phi_N^{-1}\|_\infty=\mathcal O(1)$.
Further, assume that
$\{\mathcal F_N\}$ in
(\ref{5integraloper})
is a collectively compact family of pointwise convergent operators on
$C(\Omega)$.
Finally, assume that
$(\lambda-\mathcal F)u=f$ is uniquely solvable for all
$f\in C(\Omega)$. Then for all sufficiently large $N$, say $N\geq N_0$, the operator
$(\lambda-\mathcal F_N\mathcal P_N)^{-1}$
exists and it is uniformly bounded. In addition, for the iterative solution
$v_N$
for equation
$(\lambda-\mathcal F_N\mathcal P_N)v_N=f$
we have
\begin{equation}\label{5finalerr1}
\|u-v_N\|_\infty\leq c_I \big\{\|\mathcal F_Nu-\mathcal Fu\|_\infty +
c_F(1+c_P)Ch_{X,\Omega}^{m+1}\|u\|_{W^{m+1}_{\infty}(\Omega)}\big\},
\end{equation}
provided that $u\in W^{m+1}_{\infty}(\Omega)$,
and for the discrete collocation solution
$u_N$
of equation
$(\lambda-\mathcal P_N\mathcal F_N) u_N=\mathcal P_N f$
we have
\begin{equation}\label{5finalerr2}
\|u- u_N\|_\infty\leq c_I\big \{c_P\|\mathcal F_Nu-\mathcal Fu\|_\infty +
(1+c_Pc_F)c_F(1+c_P)Ch_{X,\Omega}^{m+1}\|u\|_{W^{m+1}_{\infty}(\Omega)}\big\},
\end{equation}
where
$c_I<\infty$
is a bound for
$(\lambda-\mathcal F_N\mathcal P_N)^{-1}$.
\end{Theorem}
\begin{proof}
According to the assumptions and using Lemma
\ref{5lemmaknpn}
we conclude that
$\{\mathcal F_N\mathcal P_N\}$
is a collectively compact family of pointwise convergent operators.
By Lemma
\ref{5lemmakkn}
and the discussions after that, we have
$$
\|(\mathcal F-\mathcal F_N\mathcal P_N)\mathcal F_N\mathcal P_N\|\rightarrow 0.
$$
Thus, \eqref{5thmsteq1} is satisfied for $N\geq N_0$ if we
insert
$\mathcal T=\mathcal F$
and
$ \mathcal S=\mathcal F_N\mathcal P_N$
in to Theorem \ref{5thmst}.
Since
$(\lambda-\mathcal F)u=f$ is uniquely solvable we have
$(\lambda-\mathcal F)^{-1}\leq c_0<\infty$.
On the other hand
$\|\mathcal F_N\mathcal P_N\|\leq c_Pc_F$.
Consequently,
(\ref{5thmsteq2})
implies
$$\|(\lambda-\mathcal F_N\mathcal P_N)^{-1}\|\leq \sup_{N\geq N_0}
\frac{1+c_0c_Pc_F}{|\lambda|-c_0\|(\mathcal F-\mathcal F_N\mathcal P_N)\mathcal F_N\mathcal P_N\|}:=c_I<\infty,
$$
which proves the first assertion.
If we set
$w=u$
and
$z=v_N$
in (\ref{5thmsteq3}) then
\begin{equation*}
\begin{split}
\|u-v_N\|_\infty &\leq \|(\lambda-\mathcal F_N\mathcal P_N)^{-1}\|\|\mathcal Fu-\mathcal F_N\mathcal P_Nu\|_\infty\\
&\leq c_I\big\{\|\mathcal Fu-\mathcal F_Nu\|_\infty+\|\mathcal F_N(u-\mathcal P_Nu)\|_\infty \big\}\\
&\leq c_I\big\{\|\mathcal Fu-\mathcal F_Nu\|_\infty+c_F\|u-\mathcal P_Nu\|_\infty \big\}\\
&\leq c_I\big\{\|\mathcal Fu-\mathcal F_Nu\|_\infty+c_F(\|u-\wh u\|_\infty+\|\mathcal P_Nu-\mathcal P_N\wh u\|_\infty) \big\}\\
&\leq c_I\big\{\|\mathcal Fu-\mathcal F_Nu\|_\infty+c_F(1+c_P)\|u-\wh u\|_\infty \big\}\\
&\leq c_I\big\{ \|\mathcal Fu-\mathcal F_Nu\|_\infty +
c_F(1+c_P)Ch_{X,\Omega}^{m+1}\|u\|_{W^{m+1}_{\infty}(\Omega)}\big\}.
\end{split}
\end{equation*}
The last inequality is implied by
(\ref{LPRerrorLp}).
Moreover,
\begin{align*}
u-u_N&=u-\mathcal P_Nv_N=(u-\mathcal P_Nu)+\mathcal P_N(u-v_N),\\
\|u-u_N\|_\infty&\leq \|u-\mathcal P_Nu\|_\infty+c_P\|u-v_N\|_\infty
\end{align*}
lead to
(\ref{5finalerr2}).
\end{proof}
Theorem \ref{5thmfinal} shows that, both the quadrature and the MLS approximation error bounds affect the final estimation.
If for a sufficiently smooth kernel $\kappa(x,s)$ a high order quadrature is employed then the total error is dominated by the error of the MLS approximation.
For numerical results see \cite{mirzaei-dehghan:2010-1}.

\end{document}